\documentclass[11pt]{amsart}
\usepackage{eucal,fullpage,times,amsmath,amsthm,amssymb,mathrsfs,stmaryrd}
\usepackage[all]{xy}
\usepackage{url}

\newcommand{\calQ}{\mathcal{Q}}

\newcommand{\Z}{\mathbf{Z}}

\newcommand{\Q}{\mathbf{Q}}

\newcommand{\Spec}{{\mathrm{Spec}}}

\newcommand{\Hom}{\mathrm{Hom}}

\newcommand{\Ext}{\mathrm{Ext}}
\newcommand{\Tor}{\mathrm{Tor}}

\newcommand{\Mod}{\mathrm{Mod}}
\newcommand{\Ann}{\mathrm{Ann}}
\newcommand{\ob}{\mathrm{ob}}

\newcommand{\R}{\mathrm{R}}
\renewcommand{\L}{\mathrm{L}}

\newcommand{\id}{\mathrm{id}}

\newcommand{\coker}{\mathrm{coker}}

\newcommand{\fm}{\mathfrak{m}}

\newcommand{\comment}[1]{}

\DeclareMathOperator{\colim}{colim}

\begin{document}

\bibliographystyle{alpha}

\newtheorem{theorem}{Theorem}[section]
\newtheorem*{theorem*}{Theorem}
\newtheorem*{condition*}{Condition}
\newtheorem*{definition*}{Definition}
\newtheorem{proposition}[theorem]{Proposition}
\newtheorem{lemma}[theorem]{Lemma}
\newtheorem{corollary}[theorem]{Corollary}
\newtheorem{claim}[theorem]{Claim}
\newtheorem{claimex}{Claim}[theorem]

\theoremstyle{definition}
\newtheorem{definition}[theorem]{Definition}
\newtheorem{question}[theorem]{Question}
\newtheorem{remark}[theorem]{Remark}
\newtheorem{example}[theorem]{Example}
\newtheorem{condition}[theorem]{Condition}
\newtheorem{warning}[theorem]{Warning}
\newtheorem{notation}[theorem]{Notation}

\title{Almost direct summands}
\author{Bhargav Bhatt}
\begin{abstract}
In this short note, we point out how some new cases of the direct summand conjecture can be deduced from fundamental theorems in $p$-adic Hodge theory due to Faltings. The cases tackled include the ones when the ramification of the map being considered lies entirely in characteristic $p$.
\end{abstract}
\maketitle

The direct summand conjecture of Hochster asserts that any module-finite extension $R \to S$ of commutative rings with $R$ regular is a direct summand as an $R$-module map. This conjecture is known when $R$ contains a field (see \cite{HochsterDSC}), or if $\dim(R) \leq 3$ (see \cite{HeitmannDSC3}). The general mixed characteristic case remains wide open, and is a fundamental open problem in commutative algebra; we refer the interested reader to \cite{HochsterSurvey} for further information. 

Our goal in this note is to explain how Faltings' theory of almost \'etale extensions in $p$-adic Hodge theory can be used to prove some new cases of the direct summand conjecture. Most notably, this method allows us handle the case when the ramification lies entirely in positive characteristic. To state our result precisely, we fix a prime number $p$, and let $V$ be a complete $p$-adic discrete valuation ring whose residue field $k$ satisfies $[k:k^p] < \infty$; for example, we could take $V$ to be a finite extension of $\Z_p$. We will prove the following:

\begin{theorem}
\label{dscetalesncd}
Let $R$ be a smooth $V$-algebra, and let $f:R \to S$ be the normalisation of $R$ in a finite extension of its fraction field. Assume that there exists an \'etale map $V[T_1,\dots,T_d] \to R$ such that $f \otimes_R R[\frac{1}{p \cdot T_1 \cdot \dots \cdot T_d}]$ is unramified. Then $f:R \to S$ is a direct summand as an $R$-module map.
\end{theorem}

This note is organised as follows. In \S \ref{sec:almostartrev}, we review the basics of almost ring theory and recall Faltings' almost purity result, the main ingredient of our proof of Theorem \ref{dscetalesncd}. We then prove the theorem in \S \ref{sec:almostmainthm}. 

\section{Review of almost ring theory}
\label{sec:almostartrev}
Our proof uses almost ring theory as discovered by Tate in \cite{Tatepdivgrps}, and developed by Faltings in \cite{FaltingspHT} and \cite{FaltingsAEE} with $p$-adic Hodge theoretic applications in mind. The book \cite{GabberRameroART} provides a systematic treatment of almost ring theory, while \cite{OlssonFaltingsAEE} provides a detailed and comprehensible presentation of the arithmetic applications of Faltings' ideas. We review below the aspects of this theory most relevant to the proof of Theorem \ref{dscetalesncd}, deferring to the other sources for proofs.

\subsection{Almost mathematics}

Let $\overline{V}$ denote a valuation ring whose value group $\Lambda$ is dense in $\Q$, and let $\fm \subset \overline{V}$ be the maximal ideal. Note that $\fm$ is necessarily not finitely generated. For each non-negative $\alpha \in \Lambda$, let $\fm_\alpha$ be the (necessarily principal) ideal of elements of valuation at least $\alpha$, and let $\pi \in \overline{V}$ denote an element of valuation $1$.

\begin{example}
Let $K$ denote the fraction field of a complete $p$-adic discrete valuation ring $V$, and let $\overline{V}$ denote the integral closure of $V$ in a fixed algebraic closure $\overline{K}$ of $K$. Then $\overline{V}$ is a valuation ring whose value group is $\Q$ and, consequently, almost ring theory applies; this will be the only example relevant to us. 
\end{example}

The maximal ideal $\fm$ can be thought of as the ideal of elements with positive valuation. As the value group is dense, it follows that $\fm^2 = \fm$. This observation implies that the category $\Sigma$ of $\fm$-torsion $\overline{V}$-modules is a Serre subcategory of the abelian category $\Mod(\overline{V})$ of $\overline{V}$-modules. We refer to modules in $\Sigma$ as {\em almost zero modules}. By general nonsense, we may form the quotient abelian category 
\[\Mod(\overline{V})^a := \Mod(\overline{V})/\Sigma\]
of {\em almost} $\overline{V}$-modules. We denote the localisation functor by $M \mapsto M^a$. With this notation, we have the following description of maps in $\Mod(\overline{V})^a$ (see \S 2.2.4 of \cite{GabberRameroART}):
\[ \Hom_{\Mod(\overline{V})^a}(M^a,N^a) = \Hom_{\Mod(\overline{V})}(\fm \otimes_{\overline{V}} M,N). \]

As $\Sigma \subset \Mod(\overline{V})^a$ is closed under tensor products, the quotient $\Mod(\overline{V})^a$ inherits the structure of a symmetric $\otimes$-category with the quotient map $\Mod(\overline{V}) \to \Mod(\overline{V})^a$  being a symmetric $\otimes$-functor. This formalism allows one to systematically define ``almost analogs'' of standard notions of ring theory and, indeed, develop ``almost algebraic geometry.'' Informally, we may think of almost algebraic geometry as the study of algebraic geometry over $\overline{V}$ where all the results hold up to $\fm$-torsion. To see this program carried out in the appropriate level of generality, we suggest \cite{GabberRameroART}. We will adopt the more pragmatic stance of explaining the notions we need to precisely state Faltings' almost purity theorem. We start with the following set of definitions, borrowed from \cite{OlssonFaltingsAEE}, which allow us to define the fundamental notion of an almost \'etale morphism:

\begin{definition}
\label{arthomprop}
Let $A$ be a $\overline{V}$-algebra, and let $M$ be an $A$-module. We say that
\begin{enumerate}
\item $M$ is {\em almost projective} if $\Ext^i_A(M,N)$ is almost zero for all $A$-modules $N$ and $i > 0$.
\item $M$ is {\em almost flat} if $\Tor_i^A(M,N)$ is almost zero for all $A$-modules $N$ and $i > 0$.
\item $M$ is {\em almost faithfully flat} if it is almost flat and if for any $A$-modules $N_1$ and $N_2$, the natural map
\[ \Hom_R(N_1,N_2) \to \Hom_R(N_1 \otimes M,N_2 \otimes M) \]
has an almost zero kernel.
\item $M$ is {\em almost finitely generated} if for every strictly positive $\alpha \in \Lambda$, there exists a finitely generated $A$-module $N_\alpha$ and a $\pi^\alpha$-isomorphism $N_\alpha \simeq M$, i.e., there are maps $\phi_\alpha:N_\alpha \to M$ and $\psi_\alpha:M \to N_\alpha$ such that $\phi_\alpha \circ \psi_\alpha = \pi^\alpha \circ \id$ and $\psi_\alpha \circ \phi_\alpha = \pi^\alpha \circ \id$.
\end{enumerate}
\end{definition}

\begin{remark}
The properties defined in Definition \ref{arthomprop} are all invariant under almost isomorphisms and, consequently, depend only on the almost isomorphism class $M^a \in \Mod(\overline{V})^a$. This is clear for flatness and projectivity by the exactness of the localisation functor $\Mod(\overline{V}) \to \Mod(\overline{V})^a$. The issue of finite generation is more delicate, and we refer the interested reader to \S 2.3 of \cite{GabberRameroART} for a detailed discussion. We simply point out here that $\fm \subset \overline{V}$ is almost finitely generated with the definition given above.
 \end{remark}

\begin{warning}
As all the properties defined in Definition \ref{arthomprop} are invariant under almost isomorphisms, it is tempting to define notions such as almost projectivity purely in terms of the internal homological algebra of the abelian category $\Mod(\overline{V})^a$, i.e., in terms of the internal $\Ext$ functors. However, this approach suffers from two defects. First, as the category $\Mod(\overline{V})^a$ lacks enough projectives (the generating object $\overline{V}^a$ is not projective), one is forced to resort to a Yoneda definition of the $\Ext$ groups which is typically harder to work with. More seriously, as the Yoneda definition pays no attention to the $\otimes$-structure, the resulting theory does not interact well with the $\otimes$-structure. 
\end{warning}

Once we have access to a good theory of flatness and finite generation, one can copy the standard notions in algebraic geometry to arrive at the fundamental notion of an almost \'etale morphism.

\begin{definition}
A morphism $A \to B$  of $\overline{V}$-algebras is called an {\em almost \'etale covering} if 
\begin{enumerate}
\item $B$ is almost finitely generated, almost faithfully flat, and almost projective as an $A$-module.
\item $B$ is almost finitely generated and almost projective as a $B \otimes_A B$-module.
\end{enumerate}
\end{definition}

\begin{example}
\label{tatedim1purity}
We discuss the first non-trivial example of an almost \'etale morphism, which was discovered by Tate in his study \cite{Tatepdivgrps} of $p$-divisible groups; we follow Faltings' exposition from \cite[Theorem 1.2]{FaltingspHT}. Let $V$ be a finite extension of $\Z_p$, and let $\overline{V}$ be the normalisation of $V$ in a fixed  algebraic closure of its fraction field. Fix a tower $V = V_0 \subset V_1 \subset \dots \subset V_n \subset \dots $ of normal subextensions of $\overline{V}$ such that $\Omega^1_{V_{n+1}/V_n}$ has $V_{n+1}/p$ as a quotient for each $n$; such towers can be produced starting with a totally ramified $\Z_p$-extension of $V$, and an explicit example may be constructed starting with the tower of extensions obtained by adjoining $p$-power roots of $1$. We set $V_\infty = \colim_n V_n$. Then $V_\infty$ can be shown to be a valuation ring whose value group is dense in $\Q$ and, consequently, almost ring theory applies. The key observation (due to Tate, in different language) is that the natural map $V_\infty \to \overline{V}$ is almost \'etale (see \cite[\S 3.2, Proposition 9]{Tatepdivgrps}). A consequence of this fact is that many computations over $\overline{V}$ (such as those of certain Galois cohomology groups) can be reduced to computations over $V_\infty$, and the latter are typically much more tractable as $V_\infty$ can be chosen to have properties adapted to the problem at hand. The higher dimensional version of such arguments plays a key role in \cite{FaltingspHT}.
\end{example}

Finally, we record a basic fact concerning the almost analog of finite flat morphisms that is used in the proof of Theorem \ref{dscetalesncd}.

\begin{lemma}
\label{almostfiniteflatquot}
Let $f:A \to B$ be an inclusion of $\overline{V}$-algebras. Assume that $f$ makes $B$ an almost projective and almost faithfully flat $A$-module. Then the cokernel $\coker(f)$ is an almost projective $A$-module.
\end{lemma}
\begin{proof}
For any two $A$-modules $M$ and $N$, we have an identification of functors
\[\R\Hom(M \otimes^\L N,-) \simeq \R\Hom(M,\R\Hom(N,-)). \]
A spectral sequence argument then shows that if $M$ and $N$ are almost projective, so is $M \otimes^\L N$. If, in addition, one of $M$ or $N$ is also almost flat, then one has an almost isomorphism $M \otimes^\L N \simeq M \otimes N$. Hence, if both $M$ and $N$ are almost projective and one of them is almost flat, then $M \otimes N$ is almost projective. In particular,  we see that $B \otimes_A B$ is almost projective. Now note that we have an exact sequence
\[ 0 \to A \to B \to \coker(f) \to 0. \]
Tensoring this over $A$ with $B$ gives new exact sequence
\[ 0 \to B \to B \otimes_A B \to \coker(f) \otimes_A B \to 0.\]
The multiplication map on $B$ splits this exact sequence. Thus, $\coker(f) \otimes_A B$ is a direct summand of the almost projective $A$-module $B \otimes_A B$ and, consequently, almost projective itself. By \cite[Lemma 4.1.5]{GabberRameroART}, it follows that $\coker(f)$ is also an almost projective $A$-module.
\end{proof}

\subsection{Faltings' purity theorem}
\label{Faltingsapt}

We now state the version of Faltings' almost purity theorem most relevant to Theorem \ref{dscetalesncd}: the case of good reduction with ramification supported along a simple normal crossings divisor. There exist more general statements than the one stated below; notably, the case of semistable reduction is known.  We refer the reader to \cite[\S 2b]{FaltingsAEE} for the precise statements with proofs.

Let $V$ be a complete $p$-adic discrete valuation ring whose residue field $k$ satisfies $[k:k^p] < \infty$. Let $\overline{V}$ be the normalisation of $V$ in a fixed algebraic closure of its fraction field. We will work with almost ring theory over $\overline{V}$. Let $R$ be a smooth $V$-algebra such that $R_{\overline{V}} := R \otimes_V \overline{V}$ is a domain. Assume we are given an \'etale morphism  $V[T_1,\cdots,T_d] \to R$; such a chart always exists Zariski locally on $\Spec(R)$. Then we define a sequence of $R_{\overline{V}}$-algebras by
\[ R_n = \overline{V}[T_i^{\frac{1}{n!}}] \otimes_{V[T_i]} R. \]
One can check that the $\overline{V}$-algebras $R_n$ are finitely presented and smooth over $\overline{V}$ and, in particular, normal. Moreover, by construction, there are natural maps $R_n \to R_m$ for $n \leq m$, and we set $R_\infty = \colim_n R_n$ where the colimit is taken along the preceding maps. Note that $R_{\overline{V}} \to R_n$ is finite and faithfully flat, and consequently, $R_{\overline{V}} \to R_\infty$ is integral and faithfully flat. An elementary calculation with differentials also shows that the composite map $R \to R_{\overline{V}} \to R_\infty$ is unramified away from the divisor defined by the function $pT_1\cdots T_d$. The purity theorem says that $R \to R_\infty$ is the maximal extension of $R$ with this last property provided one works in the almost category.

\begin{theorem}[Faltings]
\label{almostpuritythm}
Let $f:R \to S$ be the normalisation of $R$ in a finite extension of its fraction field. Assume that the induced map $f \otimes_R R[\frac{1}{pT_1\cdots T_d}]$ is \'etale. If $S_n$ denotes the normalisation of $S \otimes_R R_n$, and $S_\infty = \colim_n S_n$, then the induced map $R_\infty \to S_\infty$ is an almost \'etale covering.
\end{theorem}

\begin{remark}
Theorem \ref{almostpuritythm} can be thought of as a mixed characteristic analog of Abhyankar's lemma without tameness restrictions. Recall that Abhyankar's lemma asserts (see \cite[Expos\'e XIII, Proposition 5.2]{SGA1}) that for any regular local ring $R$, the maximal extension of $R[T_1,\cdots,T_d]$ tamely ramified along the divisor associated to $T_1\cdots T_d$ may be obtained by adjoining all $n$-power roots of the $T_i$'s, where $n$ runs through integers invertible on $R$. The purity theorem does away with the tameness restrictions at the expense of only {\em almost} describing the maximal extension ramified along a normal crossings divisor in mixed characteristic. Most notably, it allows for one of the parameters $T_i$ to be replaced by $p$.
\end{remark}

\section{Proof of Theorem \ref{dscetalesncd}}
\label{sec:almostmainthm}

Our goal in this section is to prove Theorem \ref{dscetalesncd}. Correspondingly, let $V$ be as in Theorem \ref{dscetalesncd}, and let $\overline{V}$ be the normalisation of $V$ in a fixed algebraic closure of its fraction field. An essential idea informing the construction of almost ring theory is that the passage from algebraic geometry over $V$ to almost algebraic geometry over $\overline{V}$ is fairly faithful. We record one manifestation of this idea that will be useful to us.

\begin{lemma}
\label{almostzeroelts}
Let $R$ be a flat $V$-algebra essentially of finite type, and let $R_\infty$ a faithfully flat $R_{\overline{V}}$-algebra. If $M$ is an $R$-module such that $M \otimes_R R_\infty$ is zero in $\Mod(\overline{V})^a$, then $M = 0$. 
\end{lemma}
\begin{proof}
Let $x \in M$ be a non-zero element. The assumption that $M \otimes_R R_\infty$ is almost zero implies that $\fm R_\infty \subset \Ann(x \otimes 1)$. Thus, the ideal $\Ann(x \otimes 1) \subset R_\infty$ contains arbitrarily small $p$-powers. On the other hand, by the flatness of $R \to R_\infty$, we see that $\Ann(x \otimes 1) = \Ann(x) \otimes_R R_\infty$. As $I = \Ann(x)$ is an ideal in $R$ which is essentially of finite type over $V$, the smallest power $a$ of $p$ it contains is bounded above $0$ since $x \neq 0$. We formalize this by saying the natural map $V \to R/I$ defined by $1$ induces an injection 
\[ V/p^a \hookrightarrow R/I. \]
for some rational number $a > 0$. Base changing along the flat map $V \to \overline{V}$ gives an injection
\[ \overline{V}/p^a \hookrightarrow (R/I) \otimes_V \overline{V} \simeq (R/I) \otimes_R R_{\overline{V}} \simeq R_{\overline{V}} / (I \otimes_R R_{\overline{V}}). \]
also induced by $1$. The assumption that $R_{\overline{V}} \to R_\infty$ is faithfully flat then shows that we have a composite injection induced by $1$ as follows:
\[ \overline{V}/p^a \hookrightarrow R_{\overline{V}}/(I \otimes_R R_{\overline{V}}) \hookrightarrow R_\infty / (I \otimes_R R_\infty). \]
Since the ideal $I \otimes_R R_\infty = \Ann(x \otimes 1)$ contains arbitrarily small $p$-powers, so does its preimage under $\overline{V} \to R_\infty$. Hence, the map considered above is an injection only when  $a = 0$, which contradicts the assumption that $a > 0$.
\end{proof}

We are now in a position to prove Theorem \ref{dscetalesncd}.

\begin{proof}[Proof of Theorem \ref{dscetalesncd}]
After replacing $V$ with its normalisation in a suitable finite extension, and $R$ and $S$ by the corresponding base changes, we may assume that $R \otimes_V \overline{V}$ is a domain. Our goal then is to show that the exact sequence
\begin{equation}
\label{exseq1}
0 \to R \to S \to \calQ \to 0
\end{equation}
is split, where $\calQ$ is defined to be the cokernel. The obstruction to the existence of a splitting is an element $\ob(f) \in \Ext^1_R(\calQ,R)$. We will show this class almost vanishes after a suitable big extension, and then appeal to Lemma \ref{almostzeroelts}.

The assumptions in the theorem imply that there exists an \'etale morphism $V[T_1,\cdots,T_d] \to R$ such that $R \to S$ is \'etale over $R[\frac{1}{pT_1\cdots T_d}]$. Using this presentation, we define rings $R_n$, $S_n$, $R_\infty$, and $S_\infty$ as in \S \ref{Faltingsapt}. The picture over $R_\infty$ can be summarised as:
\[ \xymatrix{0 \ar[r] & R_\infty \ar@{=}[d] \ar[r] & S \otimes_R R_\infty \ar[d] \ar[r] & \calQ \otimes_R R_\infty \ar[d] \ar[r] & 0 \\
			 0 \ar[r] & R_\infty \ar[r] & S_\infty \ar[r] & \calQ_\infty \ar[r] & 0. } \]
Here the first row is obtained by tensoring the exact sequence (\ref{exseq1}) with $R_\infty$,  while $\calQ_\infty$ is the cokernel of $R_\infty \to S_\infty$. Theorem \ref{almostpuritythm} implies that the map $R_\infty \to S_\infty$ is an almost \'etale covering. By Lemma \ref{almostfiniteflatquot}, the quotient $\calQ_\infty$ is an almost projective $R_\infty$-module. Hence, the second exact sequence is almost zero when viewed as an element of $\Ext^1_{R_\infty}(\calQ_\infty,R_\infty)$. By the commutativity of the diagram, the first exact sequence then defines an almost zero element of $\Ext^1_{R_\infty}(\calQ \otimes_R R_\infty,R_\infty)$. On the other hand, by the flatness of $R \to R_\infty$, we also know that this element is simply $\ob(f) \otimes 1$ under the natural isomorphism $\Ext^1_R(\calQ,R) \otimes_R R_\infty \simeq \Ext^1_{R_\infty}(\calQ \otimes_R R_\infty,R_\infty)$. By construction of $R_\infty$, we see that $R_{\overline{V}} \to R_\infty$ is faithfully flat. By Lemma \ref{almostzeroelts} applied to the submodule of $\Ext^1_R(\calQ,R)$ generated by $\ob(f)$, we see that $\ob(f) = 0$, as desired.
\end{proof}

\begin{remark}
Theorem \ref{dscetalesncd} was stated in the context of good reduction, i.e., the base ring $R$ was assumed to be smooth over the discrete valuation ring $V$. We believe that the proof given above can generalised to handle the case of semistable reduction, the key point being that the result discussed in \S \ref{Faltingsapt} is available in the setting of semistable reduction as well; we leave the details to the interested reader.
\end{remark}

\bibliography{mybib}

\end{document}